\documentclass[a4paper, 12pt]{amsart}

\usepackage[english]{babel}
\usepackage[T1]{fontenc}
\usepackage{accents}
\usepackage{amssymb}
\usepackage{amsthm}
\usepackage{amsmath}
\usepackage{mathtools} 
\usepackage{color}
\usepackage{enumerate}
\usepackage[shortlabels]{enumitem}
\usepackage{float}
\usepackage[export]{adjustbox}
\usepackage{hyperref}
\hypersetup{
	colorlinks=true,
	linkcolor=blue, %
	filecolor=magenta, 
	urlcolor=cyan,
	citecolor= blue  %
}
\usepackage[textheight=633pt, inner=2.5cm, outer=2.5cm, voffset= -1.3cm,tmargin=124pt]{geometry}
\usepackage[most]{tcolorbox}
\tcbuselibrary{theorems}

\usepackage{comment} 
\usepackage{etoolbox}

\newcommand{\N}{\mathbb{N}}
\newcommand{\Rplus}{\mathbb{R}_{+}}
\newcommand{\BMO}{\mathrm{BMO}}
\theoremstyle{plain}
\newtheorem{theorem}{Theorem}[section]
\newtheorem{corollary}[theorem]{Corollary}
\newtheorem{lemma}[theorem]{Lemma}
\newtheorem{proposition}[theorem]{Proposition}

\makeatletter
\renewcommand*\l@subsection{\@tocline{2}{0pt}{2.8em}{2.8em}{}}
\makeatother

\theoremstyle{definition}
\newtheorem{definition}[theorem]{Definition}
\newtheorem{conjecture}{Conjecture}

\mathtoolsset{showonlyrefs}

\title[Sharp Reverse Hardy Inequality in BMO]{The Sharp Reverse Hardy Inequality in BMO for Nonincreasing Functions}

\author[A. Caldera]{A. Caldera}

\thanks{The author was supported by MICIU (Spain) under Grant FPU24/03036.}
\date{}

\begin{document}

\begin{abstract}
    Let \(Hf(x)=x^{-1}\int_0^x f(t)\,dt\) be the Hardy operator on
    \(\mathbb R_+\).  Korenovskii proved that
    \begin{equation}
       \Vert Hf\Vert_{\BMO}\geq \frac{e\alpha_0}{4}\Vert f\Vert_{\BMO},
    \end{equation}
    for every nonincreasing and locally integrable $f$, where $\alpha_0$ is defined by the relation $\Vert H\chi_{(0,1)}\Vert_{\BMO} = \alpha_0 \Vert \chi_{(0,1)}\Vert_{\BMO}$, and conjectured that the factor \(e/4\) could be removed.  We prove this conjecture by showing that every nonincreasing locally integrable $f$satisfies 
    \begin{equation}
       \Vert Hf\Vert_{\BMO}\geq \alpha_0\Vert f\Vert_{\BMO}.
    \end{equation}
    The constant $\alpha_0$ is optimal, with equality for the one-jump functions $\chi_{(0,a)}$.
\end{abstract}

\maketitle
%\tableofcontents

\section{Introduction}
The determination of optimal constants for Hardy-type operators on
monotonicity cones has attracted renewed attention in recent years.  For
nonnegative nonincreasing functions, reverse Hardy inequalities go back to Renaud \cite{Renaud1986} and Milman \cite{Milman1997}.  A closely related line of work studies the oscillation operator $H-I$, which measures the deviation of a function from its Hardy average. Sharp $L^p$ estimates for $H-I$ on the cone of nonnegative nonincreasing functions were obtained in \cite{BozaSoria2011, Kolyada2014,  KruglyakSetterqvist2008}, and were further developed in weighted form and on larger positivity cones in
\cite{BozaSoria2019,Strzelecki2020}. The corresponding discrete questions for the Ces\`aro operator, as well as sharp estimates for the adjoint Hardy and Copson operators, have recently been considered in
\cite{BenSaidBozaSoria2026, BenSaidBozaSoriaReverse2026, BenSaidSinnamon2026, BozaSoria2023, Sinnamon2022}. Besides their intrinsic interest, these results emphasize that restricting an averaging operator to a monotonicity cone may restore lower bounds that are false on the ambient space, and may lead to explicit extremizers and best constants.

The present paper concerns the endpoint counterpart of this circle of
problems. Let
\begin{equation}
    Hf(x)=\frac1x\int_0^x f(t)\,dt,\qquad x>0,
\end{equation}
be the Hardy operator on $\Rplus$.  Xiao \cite{Xiao2000} proved that, on
the cone of nonincreasing functions, the BMO seminorm of $f$ is controlled by that of $Hf$.  Korenovskii \cite{Korenovskii2002} subsequently obtained sharp BMO-BLO estimates and showed that, if
\begin{equation}
    A=\inf_{\substack{f\in L^1_{\mathrm{loc}}(\Rplus), \ f\downarrow\\
                     f\not\equiv\mathrm{constant}}}
       \frac{\Vert Hf\Vert_{\BMO}}{\Vert f\Vert_{\BMO}},
\end{equation}
then
\begin{equation}
    \frac{e\alpha_0}{4}\leq A\leq\alpha_0, \qquad \alpha_0= \frac{\Vert H\chi_{(0,1)}\Vert_{\BMO}}{\Vert \chi_{(0,1)}\Vert_{\BMO}}.
\end{equation}
He conjectured that the factor $e/4$ in the lower estimate could be
removed, or equivalently that the one-jump function $\chi_{(0,1)}$ already determines the optimal constant.

Our main result proves this conjecture.

\begin{theorem}\label{thm:main-introduction}
    Let $f$ be a nonincreasing locally integrable function on $\Rplus$.  Then
    \begin{equation}
        \Vert Hf\Vert_{\BMO}\geq\alpha_0\Vert f\Vert_{\BMO}.
    \end{equation}
    The constant $\alpha_0$ is optimal.  More explicitly, if \(\gamma_1>1\)
    is the nontrivial solution of
    \begin{equation}
        \log\!\left(\frac{\gamma}{1+\log\gamma}\right)
       =\frac{\log\gamma}{1+\log\gamma},
    \end{equation}
    then
    \begin{equation}
        \alpha_0 = \frac{4\log\gamma_1} {\gamma_1(1+\log\gamma_1)} \approx0.52123637.
    \end{equation}
\end{theorem}

The proof is developed through Sections~2--4. In Section~2 we collect the basic facts about BMO on the cone of nonincreasing functions. In particular, the BMO seminorm reduces to the initial intervals $(0,t)$, and if $s$ is a crossing point of $f$ on $(0,t)$, then
\begin{equation}
    \Omega(f;(0,t)) = \frac{2s}{t}\bigl(Hf(s)-Hf(t)\bigr).
\end{equation}
This identity reduces the theorem to an intervalwise lower estimate for the difference on the right. We also recall in this section the layer-cake representation and compute $\alpha_0$ from the one-jump function.

Section~3 develops the main ingredients for the intervalwise estimate. The layer-cake representation decomposes $f$ into one-jump functions and expresses a suitable difference of two values of $H^2f$ as the integral of a nonnegative interaction kernel $K$. The contribution below the crossing scale is controlled directly. For the contribution above that scale, we construct an explicit convex minorant of $K$, allowing Jensen's inequality to compress all the remaining layers into a single mass. A two-parameter scalar inequality then provides precisely the estimate needed at the crossing point.

In Section~4 these ingredients are combined to prove the sharp local estimate. Inserting it into the mean oscillation of $Hf$ on a canonical interval yields the main theorem. The argument is sharp for $f=\chi_{(0,1)}$, which proves optimality. Together with the known upper bound, the result also gives the sharp two-sided comparison between $\Vert Hf\Vert_{\BMO}$ and $\Vert f\Vert_{\BMO}$.

Finally, Section~5 discusses two conjectural extensions. We first obtain a non-sharp reverse estimate for the discrete Ces\`aro operator and conjecture that the optimal constant on nonincreasing sequences is determined by one-jump sequences. We then consider the spaces $\BMO_p$, $1<p<\infty$, and conjecture that their sharp reverse Hardy constants are again determined by $\chi_{(0,1)}$. The main theorem settles the latter statement at the endpoint $p=1$.

\section{Preliminaries on BMO}

We begin by recalling the definition of $\BMO$ on the positive half-line and then record the properties of nonincreasing functions needed in the proof. For an interval $I\subset\Rplus=(0,\infty)$, set
\begin{equation}
   f_I=\frac1{|I|}\int_I f(x)\,dx,
   \qquad
   \Omega(f;I)=\frac1{|I|}\int_I|f(x)-f_I|\,dx,
\end{equation}
and
\begin{equation}
   \|f\|_{\BMO}=\sup_{I\subset\Rplus}\Omega(f;I).
\end{equation}

The space $\BMO(\Rplus)$ consists of all locally integrable functions for which this quantity is finite. The $\BMO$ seminorm measures local oscillation rather than the size of the function itself. In particular, it is invariant under the addition of constants and vanishes on constant functions; it therefore becomes a norm after passing to functions modulo constants. Every bounded function belongs to $\BMO$, although $\BMO$ also contains unbounded functions.

For a fixed $t > 0$, the average of $f$ over $(0,t)$ is precisely $Hf(t)$. Since $f$ is nonincreasing, this average separates the part of the interval on which $f$ lies above its mean from the part on which it lies below its mean. This motivates the following definition.

\begin{definition}
    Let $f$ be a nonincreasing locally integrable function on $\Rplus$ and $t > 0$. We say that some $0 < s \leq t$ is a \textit{crossing point for $f$ on the interval $(0,t)$} if 
    \begin{equation}
        f(s^+) \leq Hf(t) \leq f(s^-).
    \end{equation}
\end{definition}

Such a crossing point always exists, but it need not be unique. For example, if $f$ is constant on $(0,t)$, then every $s\in(0,t]$ is a crossing point. The same monotonicity also yields our first simplification: the $\BMO$ seminorm can be computed using only intervals of the form $(0,t)$.

\begin{lemma}{\cite[Lemma~1]{Korenovskii2002}}
    Let $f$ be a nonincreasing locally integrable function in $\Rplus$. Then
    \begin{equation}
        \Vert f \Vert_{\BMO} = \sup_{t > 0} \Omega(f; (0,t)).
    \end{equation}
\end{lemma}

The next lemma gives an exact formula for the mean oscillation on each such interval in terms of a crossing point.

\begin{lemma}{\label{lemma: Omega_nonincreasing}}
    Let $f$ be a nonincreasing locally integrable function on $\Rplus$ and $t > 0$. If $0 < s \leq t$ is a crossing point of $f$ on the interval $(0,t)$, then
    \begin{equation}
        \Omega(f; (0,t))  = \dfrac{2s}{t}\Big( Hf(s) - Hf(t)).
    \end{equation}
\end{lemma}

\begin{proof}
    Let $0 < s \leq t$ be the crossing point of $f$ on the interval $(0,t)$. Then
    \begin{equation}
        \begin{aligned}
            \Omega(f; (0,t)) & = \dfrac{1}{t} \int_0^t |f(x) - f_{(0,t)}| \, dx \\
            & = \dfrac{1}{t} \int_0^t |f(x) - Hf(t)| \, dx \\
            & = \dfrac{1}{t} \bigg(\int_0^s (f(x) - Hf(t)) \, dx + \int_s^t (Hf(t)-f(x)) \, dx \bigg) \\
            & = \dfrac{1}{t} \bigg(sHf(s) - sHf(t) + (t-s)Hf(t) + sHf(s) - tHf(t) \bigg) \\
            & = \dfrac{2s}{t}\Big( Hf(s) - Hf(t)).
        \end{aligned}
    \end{equation}
\end{proof}

Thus, on the cone of nonincreasing functions, the mean oscillation on $(0,t)$ is completely encoded by the difference $Hf(s)-Hf(t)$. Consequently, the proof of the reverse $\BMO$ inequality will reduce to obtaining a sharp lower estimate for this difference in terms of an appropriate oscillation of $Hf$.

To estimate these differences, we shall use a second consequence of monotonicity. For $u > 0$, let $\lambda_f(u)=\bigl| \{x>0: f(x)>u\}\bigr|$ be the distribution function of $f$. It is easy to check that for nonnegative nonincreasing functions $f$,
\begin{equation}
    f(x)=\int_0^\infty\chi_{(0,\lambda_f(u))}(x) \, du.
\end{equation}
This formula, usually called the \emph{layer cake representation} of $f$, therefore decomposes $f$ into a superposition of one-jump functions. An immediate application of Fubini's theorem also yields a similar statement for the Hardy operator,
\begin{equation}{\label{ec: layer_H_rem}}
    \begin{aligned}
        Hf(x) & = \int_0^{\infty} H\chi_{(0, \lambda_f(t))}(x) \, dt \\
        & = \int_0^{\infty} \min \bigg \{1, \dfrac{\lambda_f(t)}{x} \bigg\} \, dt \\
        & = \dfrac{1}{x} \int_{\{t > 0: \lambda_f(t) \leq x \}} \lambda_f(t) \, dt + |\{t > 0: \lambda_f(t) > x \}|.
    \end{aligned}
\end{equation}

% \begin{remark}{\label{rem: layer_H}}
%     Let $f$ be a nonnegative, nonincreasing, and locally integrable function on $\Rplus$ and $r > 0$. Using the layer cake and Fubini's theorem it follows that
%     \begin{equation}{\label{ec: layer_H_rem}}
%         \begin{aligned}
%             Hf(r) & = \dfrac{1}{r}\int_0^r f(x) \, dx \\
%             & = \dfrac{1}{r}\int_0^r \int_0^{\infty} \chi_{(0, \lambda_f(t))}(x) \, dt \, dx \\
%             & = \int_0^{\infty} \dfrac{1}{r} \int_0^r \chi_{(0, \lambda_f(t))}(x) \, dx \, dt \\
%             & = \int_0^{\infty} \min \bigg \{1, \dfrac{\lambda_f(x)}{r} \bigg\} \, dx \\
%             & = \dfrac{1}{r} \int_{\{x > 0: \lambda_f(x) \leq r \}} \lambda_f(x) \, dx + |\{x > 0: \lambda_f(x) > r \}|
%         \end{aligned}
%     \end{equation}
% \end{remark}

The preceding representation explains the special role played by interval indicators: they are the elementary building blocks of every nonnegative nonincreasing function. In particular, the function $\chi_{(0,1)}$ provides the natural candidate for the extremal profile. We now record the constant obtained from this function. Let $\gamma_1>1$ be the nontrivial solution of
\begin{equation}\label{eq:gamma-ams}
   \log\!\left(\frac{\gamma}{1+\log\gamma}\right)
   =
   \frac{\log\gamma}{1+\log\gamma}.
\end{equation}
Put
\begin{equation}
   \gamma_0=\frac{\gamma_1}{1+\log\gamma_1}, \quad \alpha_0 = \dfrac{4\log \gamma_1}{\gamma_1(1+ \log \gamma_1)}
\end{equation}
Thus \(\log\gamma_0=\log\gamma_1/(1+\log\gamma_1)\), and
\begin{equation}
   \gamma_1\approx4.648717541,\quad
   \gamma_0\approx1.83266, \quad 
   \alpha_0\approx0.52123637.
\end{equation}

The following computation, due to Korenovskii, identifies the intervals on which the $\BMO$ seminorms of $\chi_{(0,1)}$ and its Hardy average are attained. It also shows that the constant $\alpha_0$ is an upper bound for the best constant in the reverse inequality.

\begin{lemma}{\cite[Lemma 2]{Korenovskii2002}}{\label{lemma: f_chi_0,1}}
    For the function $f(x) = \chi_{(0,1)}(x)$, $x \geq 0$, the following relations are valid:
    \begin{equation}
        \begin{aligned}
            & \Vert f \Vert_{\text{BMO}} = \Omega(f; (0,2)) = \dfrac{1}{2}, \\
            & \Vert Hf \Vert_{\text{BMO}} = \Omega(Hf; (0,\gamma_1)) = \dfrac{\alpha_0}{2}.
        \end{aligned}
    \end{equation}
\end{lemma}

\section{The interaction kernel and convex compression}

We now develop the kernel and convexity estimates underlying the sharp local inequality.

\begin{lemma}{\label{lemma: prop_K}} 
    Define, for every $z > 0$, the function
    \begin{equation}
        K(z) = \dfrac{\gamma_1}{\log \gamma_1} \Big( H^2\chi_{(0,1)}(\gamma_0/z) - H^2\chi_{(0,1)}(\gamma_1/z) \Big), \quad K(0) = K(+\infty) = 0.
    \end{equation}
    Then, the following relations hold:
    \begin{enumerate}[(i)]
        \item For every $z > 0$,
        \begin{equation}\label{eq:kernel-ams}
           K(z)= \begin{cases}
              z(1-\log z), & 0<z\leq\gamma_0,\\[2mm]
              \displaystyle
              \frac{\gamma_1+z(\log z-1-\log \gamma_1)}{\log \gamma_1}, &\gamma_0\leq z\leq\gamma_1,\\[3mm] 
              0, &z\geq\gamma_1.
            \end{cases}
        \end{equation}
        
        \item For $0 < z \leq 1$, $K(z) \geq z$.
        
        \item $K|_{[1, \infty)}$ has a convex minorant $\phi$ such that $0 \leq \phi \leq K|_{[1, \infty)}$ and 
        \begin{equation}{\label{ec: ineq_phi}}
            \phi(z) \geq (2-\sqrt{z})^2, \quad 1 \leq z \leq 4.
        \end{equation}
    \end{enumerate}
\end{lemma}

\begin{proof}
    $(i)$ First, note that for every $z > 0$,
    \begin{equation}
        H^2\chi_{(0,1)}(z)= \begin{cases}
          1, & 0 < z \leq 1,\\[2mm]
          \displaystyle
          \frac{1+\log z}{z}, & z > 1.\\[3mm] 
        \end{cases}
    \end{equation}
    If $0 < z \leq \gamma_0$,
    \begin{equation}
        \begin{aligned}
            K(z) & = \dfrac{\gamma_1}{\log \gamma_1} \bigg(\frac{1+\log (\gamma_0/z)}{\gamma_0/z} - \frac{1+\log (\gamma_1/z)}{\gamma_1/z}\bigg) \\
            & = z \dfrac{\gamma_1}{\log \gamma_1} \bigg( \dfrac{1+\log \gamma_0}{\gamma_0} - \dfrac{1+\log \gamma_1}{\gamma_1} - \log z \bigg( \dfrac{1}{\gamma_0} - \dfrac{1}{\gamma_1} \bigg)\bigg) \\
            & = z \dfrac{\gamma_1}{\log \gamma_1} \bigg( \dfrac{\log \gamma_1}{\gamma_1} - \log z \dfrac{\log \gamma_1}{\gamma_1}\bigg)  = z (1- \log z).
        \end{aligned}
    \end{equation}
    On the other hand, if $\gamma_0\leq z\leq\gamma_1$,
    \begin{equation}
        K(z) = \dfrac{\gamma_1}{\log \gamma_1} \bigg( 1 - \frac{1+\log (\gamma_1/z)}{\gamma_1/z}\bigg) = \dfrac{\gamma_1+z(\log z-1- \log \gamma_1)}{\log \gamma_1}.
    \end{equation}
    The left case $z \geq \gamma_1$ is immediate.

    $(ii)$ Suppose that $0 < z \leq 1$. Then $K(z) = z(1- \log z) \geq z$.

    $(iii)$ A simple computation yields
    \begin{equation}
        K'(z)= \begin{cases}
        -\log z, & 0<z<\gamma_0,\\[2mm]
        \displaystyle
        \frac{\log z-\log\gamma_1}{\log\gamma_1},
        & \gamma_0<z<\gamma_1,\\[3mm]
        0, & z>\gamma_1.
        \end{cases}, \quad 
        K''(z)= \begin{cases}
            -\dfrac{1}{z}, & 0<z<\gamma_0,\\[3mm]
            \displaystyle
            \dfrac{1}{z\log\gamma_1},
            & \gamma_0<z<\gamma_1,\\[3mm]
            0, & z>\gamma_1.
            \end{cases}
    \end{equation}
    Therefore the function $K$ is concave on $[1, \gamma_0]$ and convex on $[\gamma_0, \infty)$. Let $\tau \in (\gamma_0, \gamma_1)$ be the point at which the line through $(1,1)$ is tangent to $K$, i.e.
    \begin{equation}
        \dfrac{K(\tau)-1}{\tau-1} = K'(\tau)
    \end{equation}
    Equivalently,
    \begin{equation}
        \tau - \log \tau = \gamma_1 - 2 \log \gamma_1, \quad \tau \approx 2.486350912.
    \end{equation}
    Set $\sigma = (\log \gamma_1 - \log \tau)/\log \gamma_1$, and define
    \begin{equation}
        \phi(y) = \begin{cases}
            1-\sigma(y-1), & 1 \leq y \leq \tau,\\[1mm]
            K(y), & y \geq \tau.
            \end{cases}
    \end{equation}
    Clearly $0 \leq \phi \leq K|_{[1, \infty)}$ and $\phi$ is convex, so $\phi$ is a convex minorant of $K|_{[1, \infty)}$. Finally, let us show that $\phi(y) \geq (2-\sqrt{y})^2$ for every $1 \leq y \leq 4$. For $\tau \leq y \leq 4$, put $G(y) = K(y) - (2- \sqrt{y})^2$. Then
    \begin{equation}
        \begin{aligned}
            G'(y) & = \dfrac{\log y}{\log \gamma_1} - 2 + \dfrac{2}{\sqrt{y}} = \dfrac{2\log \sqrt{y}}{\log \gamma_1} - 2 + \dfrac{2}{\sqrt{y}} \\
            & = \dfrac{2}{\log \gamma_1}\bigg(1-\dfrac{1}{\sqrt{y}}\bigg) \bigg( \dfrac{\sqrt{y} \log \sqrt{y}}{\sqrt{y}-1} - \log \gamma_1 \bigg).
        \end{aligned}
    \end{equation}
    Moreover the function $q(y) = (y \log y)/(y-1)$ is increasing on $(1, \infty)$, which yields
    \begin{equation}
        G'(y) \leq \dfrac{2}{\log \gamma_1}\bigg(1-\dfrac{1}{\sqrt{y}}\bigg) \bigg( 2\log 2 - \log \gamma_1 \bigg) \leq 0.
    \end{equation}
    Since $G(4) = K(4) > 0$, this proves \eqref{ec: ineq_phi} on $[\tau, 4]$. On $[1,\tau]$, the function $(2- \sqrt{x})^2$ is convex and lies below $\phi$ at both endpoints, while $\phi$ is linear. Hence \eqref{ec: ineq_phi} holds on the whole interval $[1,4]$.
\end{proof}

The estimate $K(z)\geq z$ for $0<z\leq1$ will control the layers lying below the reference scale. For $z\geq1$, the kernel is not convex on its whole support. We therefore replace it from below by the convex function $\phi$ constructed above. This is the key step that will allow us to apply Jensen's inequality and compress all the layers above the reference scale into their total mass and first moment.

After this compression, the relevant information is encoded by four scalar parameters: the contribution $a$ below the reference scale, the mass $c$ above it, the normalized first moment $m$, and the ratio $\delta$ of the two scales. The crossing-point condition will later imply $c\leq\delta(a+cm)\leq a+c$. The next lemma extracts the precise scalar inequality needed from these constraints.

\begin{lemma}{\label{lemma: scalar_lem}}
    Let $a,c\geq0$, $m\geq1$, and $0<\delta\leq1$.  If $\phi$ is the function from Lemma~\ref{lemma: prop_K} and $c\leq\delta(a+cm)\leq a+c$, then
    \begin{equation}\label{eq:scalar-goal}
       a+c\phi(m) \geq 4\delta\bigl(a+c-\delta(a+cm)\bigr).
    \end{equation}
\end{lemma}

\begin{proof}
    If $c=0$, the assertion is just
    $4\delta(1-\delta)\leq1$.  Suppose that $c>0$, and put $r=a/c$
    After division by $c$, the hypothesis becomes 
    \begin{equation}
        \dfrac{1}{r+m}\leq\delta\leq \dfrac{r+1}{r+m},
    \end{equation}
    and it remains to prove
    \begin{equation}\label{eq:r-goal}
       r+\phi(m) \geq 4\delta(r+1)-4\delta^2(r+m) = q(\delta).
    \end{equation}  
    Assume first that $r\geq1$. The quadratic $q$ attains its maximum at $\delta = (r+1)(2(r+m))^{-1}$, and so
    \begin{equation}{\label{ec: max_q}}
        q(\delta) \leq q\bigg( \dfrac{r+1}{2(r+m)}\bigg) = \frac{(r+1)^2}{r+m}.
    \end{equation}
    For every \(m\geq1\),
    \begin{equation}\label{eq:rational-bound}
       \phi(m)\geq\frac{3-m}{1+m}.
    \end{equation}
    Indeed, this is immediate from $\phi\geq0$ when $m\geq3$. For
    $1\leq m\leq3$, it follows from Lemma~\ref{lemma: prop_K} that
    \begin{equation}
        \begin{aligned}
            \phi(m) \geq (2-\sqrt m)^2 & = (2-\sqrt m)^2-\frac{3-m}{1+m} + \frac{3-m}{1+m} \\
            & = \frac{(\sqrt m-1)^4}{1+m} + \frac{3-m}{1+m} \\
            & \geq \frac{3-m}{1+m}.
        \end{aligned}
    \end{equation}
    Furthermore,
    \begin{equation}
        \frac{3-m}{1+m} \geq\frac{2r+1-rm}{r+m},
    \end{equation}
    because the difference after multiplication by the positive
    denominators is \((r-1)(m-1)^2\).  These estimates yield
    \begin{equation}
        r+\phi(m) \geq r+\frac{2r+1-rm}{r+m} = \frac{(r+1)^2}{r+m} \geq q(\delta),
    \end{equation}
    which proves \eqref{eq:r-goal} in this case.
    
    Now assume that $0\leq r\leq1$.  Since $\delta\geq1/(r+m)$,  the quadratic $q$ is decreasing on the allowed range, and hence
    \begin{equation}\label{eq:small-r-bound}
       q(\delta) \leq q\bigg( \dfrac{1}{r+m}\bigg) = \frac{4r}{r+m}.
    \end{equation}
    If $m \geq 4$, then $q(\delta) \leq r \leq r + \phi(m)$.  If
    $1\leq m\leq4$, we conclude by Lemma~\ref{lemma: prop_K} that
    \begin{equation}
        r+\phi(m) \geq r + (2-\sqrt m)^2 = \dfrac{4r+ \Big(r-\sqrt{m}(2- \sqrt{m})\Big)^2}{r+m} \geq \frac{4r}{r+m} \geq q(\delta).
    \end{equation}
    This completes the proof.
\end{proof}

It remains to connect the kernel $K$ with the original function $f$. Since every nonnegative nonincreasing function is a superposition of one-jump functions, the layer-cake representation allows us to integrate the preceding computation over its level sets. Using also the commutation of the Hardy operator with dilations, we obtain the following exact representation.

\begin{lemma}{\label{prop: properties_P}}
     Let $f$ be a  nonnegative, nonincreasing and locally integrable function on $\Rplus$, and let $r > 0$. If $K$ is the function defined in Lemma~\ref{lemma: prop_K}, then
    \begin{equation}
         H^2f(\gamma_0 r) - H^2f(\gamma_1 r)  = \dfrac{\log \gamma_1}{\gamma_1}\int_0^{\infty} K\bigg( \dfrac{\lambda_f(x)}{r} \bigg) \, dx.
    \end{equation}
\end{lemma}

\begin{proof}
    First, using the layer cake representation, Fubini's theorem and the fact that Hardy's operator commutes with dilations it is easy to see, as in \eqref{ec: layer_H_rem}, that
    \begin{equation}
        H^2f(x) = \int_0^{\infty} H^2\chi_{(0, \lambda_f(t))}(x) \, dt = \int_0^{\infty} H^2\chi_{(0, 1)}\bigg(\dfrac{x}{\lambda_f(t)} \bigg) \, dt,
    \end{equation}
    for every $x > 0$. Hence, given $r > 0$,
    \begin{equation}
         \begin{aligned}
             H^2f(\gamma_0 r) - H^2f(\gamma_1 r)  & = \int_0^{\infty} \bigg( H^2\chi_{(0, 1)}\bigg(\dfrac{\gamma_0 r}{\lambda_f(t)} \bigg) - H^2\chi_{(0, 1)}\bigg(\dfrac{\gamma_1 r}{\lambda_f(t)} \bigg)\bigg) \, dt \\
             & = \dfrac{\log \gamma_1}{\gamma_1}\int_0^{\infty} K\bigg( \dfrac{\lambda_f(x)}{r} \bigg) \, dx.
         \end{aligned}
    \end{equation}
\end{proof}

\section{The main results}

We now combine the preceding lemmas to prove the sharp local estimate and then the main BMO inequality.

\begin{theorem}{\label{thm: ineq_main_in}}
    Let $f$ be a nonincreasing locally integrable function on $\Rplus$ and $t > 0$. If $0 < s \leq t$ is a crossing point of $f$ on the interval $(0,t)$, then
    \begin{equation}
        \dfrac{4\log \gamma_1}{\gamma_1}\bigg( Hf(s) - Hf(t)\bigg) \leq \dfrac{t}{s} \bigg( H^2f(\gamma_0 s) - H^2f(\gamma_1 s)\bigg).
    \end{equation}
    Moreover, the constant $(4\log \gamma_1)/\gamma_1$ is the best possible.
\end{theorem}

\begin{proof}
    First, suppose that $f$ is nonnegative, nonincreasing and locally integrable, and choose a crossing point $s$ of $f$ on the interval $(0,t)$. By Lemma~\ref{prop: properties_P} it suffices to see that
    \begin{equation}
        \dfrac{4s}{t}\bigg( Hf(s) - Hf(t)\bigg) \leq \int_0^{\infty} K\bigg( \dfrac{\lambda_f(x)}{s} \bigg) \, dx.
    \end{equation}
    For the sake of clarity consider the sets
    \begin{equation}
        E_0 = \{x > 0: \lambda_f(x) \leq s \}, \quad E_1 = \{x > 0: s < \lambda_f(x) \leq t \},
    \end{equation}
    and define
    \begin{equation}
        a=\int_{E_0}\frac{\lambda_f(x)}s\,dx, \qquad c=|E_1|.
    \end{equation}
    If $c = 0$, since $s \leq t$ and $\sup_{0 \leq\delta \leq 1} 4\delta(1-\delta) = 1$, we have by \eqref{ec: layer_H_rem} and Lemma~\ref{lemma: prop_K},
    \begin{equation}
        \begin{aligned}
            \int_0^{\infty} K\bigg( \dfrac{\lambda_f(x)}{s} \bigg) \, dx & \geq \int_{E_0} K\bigg( \dfrac{\lambda_f(x)}{s} \bigg) \, dx \\
            & \geq \dfrac{1}{s}\int_{E_0} \lambda_f(x) \, dx \\
            & \geq 4\dfrac{s}{t} \bigg( 1- \dfrac{s}{t} \bigg)\dfrac{1}{s}\int_{E_0} \lambda_f(x) \, dx \\
            & = \dfrac{4s}{t}\bigg(\dfrac{1}{s}\int_{E_0} \lambda_f(x) \, dx - \dfrac{1}{t}\int_{E_0} \lambda_f(x) \, dx \bigg) \\
            & = \dfrac{4s}{t}\bigg(\dfrac{1}{s}\int_{E_0} \lambda_f(x) \, dx - \dfrac{1}{t}\int_{E_0 \cup E_1} \lambda_f(x) \, dx \bigg) \\
            & = \dfrac{4s}{t}\bigg(\dfrac{1}{s}\int_{\{x > 0: \lambda_f(x) \leq s \}} \lambda_f(x) \, dx - \dfrac{1}{t}\int_{\{x > 0: \lambda_f(x) \leq t \}} \lambda_f(x) \, dx \bigg) \\
            & = \dfrac{4s}{t}\bigg( Hf(s) - Hf(t)\bigg).
        \end{aligned}
    \end{equation}
    If \(c>0\), set
    \begin{equation}
        m=\frac1c\int_{E_1}\frac{\lambda_f(x)}s\,dx.
    \end{equation} 
    By Lemma~\ref{lemma: prop_K} and Jensen's inequality,
    \begin{equation}{\label{ec: Jensen_ineq}}
        \begin{aligned}
            \int_0^{\infty} K\bigg( \dfrac{\lambda_f(x)}{s} \bigg) \, dx & \geq \int_{E_0} K\bigg( \dfrac{\lambda_f(x)}{s} \bigg) \, dx + \int_{E_1} K\bigg( \dfrac{\lambda_f(x)}{s} \bigg) \, dx \\
            & \geq \dfrac{1}{s}\int_{E_0} \lambda_f(x) \, dx + \int_{E_1} \phi\bigg( \dfrac{\lambda_f(x)}{s} \bigg) \, dx \\
            & \geq \dfrac{1}{s}\int_{E_0} \lambda_f(x) \, dx + c \dfrac{1}{|E_1|}\int_{E_1} \phi\bigg( \dfrac{\lambda_f(x)}{s} \bigg) \, dx \\
            & \geq \dfrac{1}{s}\int_{E_0} \lambda_f(x) \, dx + c \phi \Bigg(\dfrac{1}{|E_1|}\int_{E_1} \dfrac{\lambda_f(x)}{s} \, dx \Bigg) \\
            & = a + c\phi(m).
        \end{aligned}
    \end{equation}
    Similarly, if we consider $\delta = s/t \leq 1$, by \eqref{ec: layer_H_rem} we have
    \begin{equation}\label{eq:hardy-difference}
       \begin{aligned}
           Hf(s)-Hf(t) & = \dfrac{1}{s} \int_{E_0} \lambda_f(x) \, dx - \dfrac{1}{t} \int_{E_0\cup E_1} \lambda_f(x) \, dx  + c \\
           & = \dfrac{1}{s} \int_{E_0} \lambda_f(x) \, dx - \delta \bigg( \dfrac{1}{s}\int_{E_0} \lambda_f(x) \, dx + \dfrac{1}{s}\int_{E_1} \lambda_f(x) \, dx \bigg)  + c \\
           & = a+c-\delta(a+cm).
       \end{aligned}
    \end{equation}
    In particular, we have
    \begin{equation}\label{eq:scalar-crossing}
       \begin{aligned}
             Hf(t)  & = \dfrac{1}{t} \int_{E_0 \cup E_1} \lambda_f(x) \, dx + |\{x > 0: \lambda_f(x) > t \}| \\
            & = \delta \bigg( \dfrac{1}{s}\int_{E_0} \lambda_f(x) \, dx + \dfrac{1}{s}\int_{E_1} \lambda_f(x) \, dx \bigg) + |\{x > 0: \lambda_f(x) > t \}| \\
            & = \delta(a+cm) + |\{x > 0: \lambda_f(x) > t \}|,
       \end{aligned}
    \end{equation}
    On the other hand, the generalized inverse relations for a nonnegative nonincreasing function are 
    \begin{equation}
        f(s^+) = |\{ x> 0: \lambda_f(x) > s \}| = c + |\{x > 0: \lambda_f(x) > t \}|,
    \end{equation}
    and
    \begin{equation}
        \begin{aligned}
            f(s^-) = |\{ x> 0: \lambda_f(x) \geq s \}| & = c + |\{x > 0: \lambda_f(x) > t \}| + |\{x > 0: \lambda_f(x) = s \}| \\
            & \leq c + |\{x > 0: \lambda_f(x) > t \}| + \int_{\{x > 0: \lambda_f(x) = s \}} \dfrac{\lambda_f(x)}{s} \, dx \\
            & \leq c + |\{x > 0: \lambda_f(x) > t \}| + a.
        \end{aligned}
    \end{equation}
    Thus, since $s$ is a crossing point of $f$ on $(0,t)$, it follows that $c \leq \delta(a+cm) \leq c+a$. Hence, we conclude by Lemma~\ref{lemma: scalar_lem} together with \eqref{ec: Jensen_ineq} and \eqref{eq:hardy-difference} that
    \begin{equation}
        \begin{aligned}
            \int_0^{\infty} K\bigg( \dfrac{\lambda_f(x)}{s} \bigg) \, dx \geq a + c\phi(m) \geq 4\delta\bigl(a+c-\delta(a+cm)\bigr) = \dfrac{4s}{t}\bigg( Hf(s) - Hf(t)\bigg).
        \end{aligned}
    \end{equation}
    This completes the proof when $f$ is nonnegative, nonincreasing and locally integrable. Now, suppose that $f$ is nonincreasing and locally integrable on $\Rplus$. Let $R > \max\{\gamma_1 s, t\}$ and $f_R = (f-f(R))_+$. For every $0 < z \leq R$, $Hf(z) = Hf_R(z) + f(R)$. Since $s$ is also a crossing point of $f_R$ on $(0,t)$ and $f_R$ is a nonnegative, nonincreasing and locally integrable function on $\Rplus$, we conclude by the previous case that
    \begin{equation}
        \begin{aligned}
            \dfrac{4\log \gamma_1}{\gamma_1}\bigg( Hf(s) - Hf(t)\bigg) & = \dfrac{4\log \gamma_1}{\gamma_1}\bigg( Hf_R(s) - Hf_R(t)\bigg) \\
            & \leq \dfrac{t}{s} \bigg( H^2f_R(\gamma_0 s) - H^2f_R(\gamma_1 s)\bigg) \\
            & = \dfrac{t}{s} \bigg( H^2f(\gamma_0 s) - H^2f(\gamma_1 s)\bigg).
        \end{aligned}
    \end{equation}
    It remains to show the optimality of the constant. Take $f = \chi_{(0,1)}$ and $t = 2$. Then
    \begin{equation}
       H\chi_{(0,1)}(x)= \begin{cases}
          1, & 0 < x \leq 1,\\[2mm]
          \displaystyle
          \frac{1}{x}, & x > 1,\\[3mm] 
        \end{cases}, \quad H^2\chi_{(0,1)}(x)= \begin{cases}
          1, & 0 < x \leq 1,\\[2mm]
          \displaystyle
          \frac{1+\log x}{x}, & x > 1.\\[3mm] 
        \end{cases}
    \end{equation}
    Clearly the crossing point of $f$ on $(0,t)$ is $s = 1$. Therefore
    \begin{equation}
        \begin{aligned}
            \dfrac{t}{s} \bigg( H^2f(\gamma_0 s) - H^2f(\gamma_1 s)\bigg) & = 2\bigg( H^2\chi_{(0,1)}(\gamma_0) - H^2\chi_{(0,1)}(\gamma_1)\bigg) \\
            & = 2 \bigg(\frac{1+\log \gamma_0}{\gamma_0} - \frac{1+\log \gamma_1}{\gamma_1}  \bigg) \\
            & =  2 \bigg(\frac{1+2\log \gamma_1}{\gamma_1} - \frac{1+\log \gamma_1}{\gamma_1}  \bigg) \\
            & = 2 \dfrac{\log \gamma_1}{\gamma_1},
        \end{aligned}
    \end{equation}
    and
    \begin{equation}
        \begin{aligned}
            Hf(s) - Hf(t) = H\chi_{(0,1)}(1) - H\chi_{(0,1)}(2) = \dfrac{1}{2}.
        \end{aligned}
    \end{equation}
    Thus the constant $(4\log \gamma_1)/\gamma_1$ is optimal.
\end{proof}

\begin{theorem}{\label{thm: main_result}}
    Suppose that $f$ is a nonincreasing locally integrable function in $\Rplus$. Then
    \begin{equation}
        \Vert Hf \Vert_{\BMO} \geq \alpha_0 \Vert f \Vert_{\BMO}.
    \end{equation}
    Moreover, the constant $\alpha_0$ is optimal.
\end{theorem}

\begin{proof}
    Let $f$ be a nonincreasing locally integrable function in $\Rplus$ and $t > 0$. Consider the crossing point $s$ of $f$ on $(0,t)$. Since $\int_0^{\gamma_1s} (Hf(x) - H^2f(\gamma_1s)) \, dx = 0$,
    \begin{equation}
        \begin{aligned}
            \Omega(Hf;(0, \gamma_1s)) & = \dfrac{1}{\gamma_1s} \int_0^{\gamma_1s} |Hf(x) - H^2f(\gamma_1s)| \, dx \\
            & = \dfrac{1}{\gamma_1s} \bigg(\int_0^{\gamma_1s} \big[Hf(x) - H^2f(\gamma_1s)\big]_+ \, dx + \int_0^{\gamma_1s} \big[Hf(x) - H^2f(\gamma_1s)\big]_- \, dx \bigg) \\
            & = \dfrac{2}{\gamma_1s} \int_0^{\gamma_1s} \big[Hf(x) - H^2f(\gamma_1s)\big]_+ \, dx \\
            & \geq \dfrac{2}{\gamma_1s} \int_0^{\gamma_0s} \big[Hf(x) - H^2f(\gamma_1s)\big]_+ \, dx \\
            & \geq \dfrac{2}{\gamma_1s} \int_0^{\gamma_0s} \Big(Hf(x) - H^2f(\gamma_1s)\Big) \, dx \\
            & \geq \dfrac{2}{\gamma_1s} \bigg(\gamma_0s H^2f(\gamma_0s) - \gamma_0s H^2f(\gamma_1s) \bigg) \\
            & \geq \dfrac{2\gamma_0}{\gamma_1} \bigg( H^2f(\gamma_0s) - H^2f(\gamma_1s) \bigg).
        \end{aligned}
    \end{equation}
    Therefore, by Lemma~\ref{lemma: Omega_nonincreasing} and Theorem~\ref{thm: ineq_main_in},
    \begin{equation}
        \begin{aligned}
            \Vert Hf \Vert_{\BMO} & \geq \Omega(Hf;(0, \gamma_1s)) \\
            & \geq \dfrac{2\gamma_0}{\gamma_1} \bigg( H^2f(\gamma_0s) - H^2f(\gamma_1s) \bigg) \\
            & \geq \dfrac{4\gamma_0\log \gamma_1}{\gamma_1^2} \dfrac{2s}{t}\bigg( Hf(s) - Hf(t)\bigg) \\
            & = \alpha_0 \Omega(f; (0,t)).
        \end{aligned}
    \end{equation}
    Taking the supremum over all possible $t > 0$ gives the inequality $\Vert Hf \Vert_{\BMO} \geq \alpha_0 \Vert f \Vert_{\BMO}$. The optimality follows by taking $f = \chi_{(0,1)}$ and applying Lemma~\ref{lemma: f_chi_0,1}. 
\end{proof}

\begin{corollary}
    Suppose that $f$ is a nonincreasing locally integrable function in $\Rplus$. Then
    \begin{equation}
       \alpha_0 \Vert f \Vert_{\BMO} \leq \Vert Hf \Vert_{\BMO}\leq \Vert f \Vert_{\BMO}.
    \end{equation}
    Moreover, the constants $\alpha_0$ and $1$ are optimal.
\end{corollary}

\begin{proof}
    This follows from Theorem~\ref{thm: main_result} and \cite{Xiao2000}. To see that the constant $1$ cannot be decreased, it suffices to consider the function $f(x) = \ln(1/x)$, $x > 0$, since $Hf(x) = f(x) + 1$.
\end{proof}

\section{Open problems}

We conclude with two natural extensions of Theorem \ref{thm: main_result}.

\subsection{The discrete setting}

For a real sequence \(a=(a_n)_{n\geq1}\) and a finite interval
\(J=\{m,m+1,\ldots,n\}\subset\mathbb N\), set
\begin{equation*}
    a_J=\frac1{|J|}\sum_{k\in J}a_k,
    \qquad
    \Omega_{\mathrm d}(a;J)
    =\frac1{|J|}\sum_{k\in J}|a_k-a_J|.
\end{equation*}
The discrete BMO space \(\BMO_{\mathrm d}(\mathbb N)\) consists of
the sequences for which
\begin{equation*}
    \Vert a\Vert_{\BMO_{\mathrm d}}
    =\sup_{J\subset\mathbb N}\Omega_{\mathrm d}(a;J)<\infty,
\end{equation*}
where the supremum is taken over all finite intervals of consecutive
integers.  The discrete Hardy, or Ces\`aro, operator is defined by
\begin{equation}
    (Ca)_n=\frac1n\sum_{k=1}^n a_k,
    \qquad n\geq1.
\end{equation}

Given a nonincreasing sequence, we associate with it a nonincreasing step function on $\Rplus$. This construction relates the discrete mean oscillation of the sequence to the $\BMO$ seminorm of the step function and allows us to apply the continuous theorem. The comparison involves a loss in the constants and gives the following preliminary estimate.

\begin{proposition}
    Let $C$ be the Cesàro operator and let $a$ be a nonnegative nonincreasing sequence. Then
    \begin{equation}{\label{ec: discr_ineq}}
        \Vert Ca \Vert_{\mathrm{BMO}_{\mathrm d}}
        \geq \frac{\alpha_0}{8} \Vert a \Vert_{\mathrm{BMO}_{\mathrm d}}.
    \end{equation}
\end{proposition}

\begin{proof}
    Define
    \begin{equation}
        f(t)=a_n,\qquad B(t)=(Ca)_n,\qquad n-1<t\leq n.
    \end{equation}
    Since intervals with integer endpoints reproduce the discrete mean oscillations,
    \begin{equation*}
        \Vert f\Vert_{\mathrm{BMO}(\mathbb R_+)}
        \geq \Vert a \Vert_{\mathrm{BMO}_{\mathrm d}}.
    \end{equation*}
    We claim that $\Vert B\Vert _{\mathrm{BMO}(\mathbb R_+)}\leq 4\Vert Ca\Vert_{\mathrm{BMO}_{\mathrm d}}$. Let $t > 0$ arbitrary. If $t \leq 1$, then $B$ is constant on $(0,t)$, so $\Omega(B; (0,t)) = 0$. For $t > 1$, let $n \geq 2$ be the smallest natural number such that $t \leq n$. Using
    \begin{equation}
        \dfrac{1}{t}\int_0^t |B(x) - HB(t)| \, dx \leq \dfrac{2}{t}\int_0^t |B(x) - HB(n)| \, dx,
    \end{equation}
    it follows that
    \begin{equation}
        \begin{aligned}
            \Omega(B; (0,t)) & \leq \dfrac{2}{t}\int_0^t |B(x) - HB(n)| \, dx \\
            & \leq \dfrac{2}{t}\int_0^n |B(x) - HB(n)| \, dx \\
            & = \dfrac{2n}{t} \Omega(B; (0,n)) \\
            & = \dfrac{2n}{t} \frac{1}{n}\sum_{k=1}^n \left| (Ca)_k-\frac{1}{n} \sum_{j=1}^n (Ca)_j \right| \\
            & \leq \dfrac{2n}{t} \Vert Ca\Vert_{\mathrm{BMO}_{\mathrm d}} \leq 4 \Vert Ca\Vert_{\mathrm{BMO}_{\mathrm d}}.
        \end{aligned}
    \end{equation}
    Taking the supremum over all possible $t > 0$ yields that $\Vert B\Vert _{\mathrm{BMO}(\mathbb R_+)}\leq 4\Vert Ca\Vert_{\mathrm{BMO}_{\mathrm d}}$. 
    On the other hand, clearly $Hf$ is nonincreasing and $Hf(n) = (Ca)_n$ for every $n \geq 1$. Thus, for $n\geq2$ and $n-1<t\leq n$,
    \begin{equation*}
        (Ca)_{n-1} \geq Hf(t)\geq (Ca)_n.
    \end{equation*}
    Since the mean oscillation of $Ca$ on $\{n-1,n\}$ is $|(Ca)_{n-1}-(Ca)_{n}|/2\leq \Vert Ca\Vert_{\mathrm{BMO}_{\mathrm d}}$, it follows that $\Vert Hf-B\Vert_\infty \leq 2\Vert Ca\Vert_{\mathrm{BMO}_{\mathrm d}}$. Hence
    \begin{equation}
        \Vert Hf\Vert_{\mathrm{BMO}}
        \leq \Vert B\Vert_{\mathrm{BMO}}+2\Vert Hf-B\Vert_\infty
        \leq 8\Vert Ca\Vert_{\mathrm{BMO}_{\mathrm d}}.
    \end{equation}
    Therefore, the result follows from Theorem~\ref{thm: main_result}, since
    \begin{equation}
        \alpha_0\Vert a\Vert_{\mathrm{BMO}_{\mathrm d}}
        \leq \alpha_0\Vert f\Vert_{\mathrm{BMO}}
        \leq \Vert Hf\Vert_{\mathrm{BMO}}
        \leq 8\Vert Ca\Vert_{\mathrm{BMO}_{\mathrm d}}.
    \end{equation}
\end{proof}

To formulate the sharp problem, we consider the discrete one-jump sequences, which are the natural analogues of the functions $\chi_{(0,a)}$. Let us denote for every $n \in \N$,
\begin{equation}
    e_n(k) = \begin{cases}
            1, & 1 \leq k \leq n,\\[1mm]
            0, & k > n,
            \end{cases} 
    \qquad \alpha_{d} = \inf_{n \in \N} \dfrac{\Vert Ce_n\Vert_{\BMO_{\mathrm d}}}{\Vert e_n \Vert_{\BMO_{\mathrm d}}}.
\end{equation}
Testing the inequality on the sequences $e_n$ shows that no lower constant larger than $\alpha_{\mathrm d}$ is possible. We conjecture that these one-jump sequences already determine the optimal constant for every nonincreasing sequence, i.e.

\begin{conjecture}%[Discrete reverse Hardy inequality] 
    For every nonincreasing sequence $a = (a_n)_{n\geq 1}$,
    \begin{equation}
        \Vert Ca\Vert_{\BMO_{\mathrm d}}
        \geq \alpha_{d}\Vert a \Vert_{\BMO_{\mathrm d}}.
    \end{equation}
\end{conjecture}

\subsection{The spaces \texorpdfstring{\(\BMO_p\)}{BMO-p}}

Let $1<p<\infty$.  For $f\in L^p_{\mathrm{loc}}(\Rplus)$, define
\begin{equation}
    \Omega_p(f;I)
    =\left(\frac1{|I|}\int_I|f(x)-f_I|^p\,dx\right)^{1/p},
\end{equation}
and
\begin{equation*}
    \Vert f \Vert_{\BMO_p} =\sup_{I\subset\Rplus}\Omega_p(f;I).
\end{equation*}
The space $\BMO_p(\Rplus)$ consists of all locally $p$-integrable
functions with finite $\BMO_p$ seminorm.  The John-Nirenberg
inequality \cite{JohnNirenberg} and Theorem \ref{thm: main_result} give a reverse estimate with a non-sharp constant depending on $p$. This argument, however, does not yield the optimal constant. The problem is therefore to identify the best possible one. Set
\begin{equation}
    \alpha_p =\frac{\|H\chi_{(0,1)}\|_{\BMO_p}}{\|\chi_{(0,1)}\|_{\BMO_p}}.
\end{equation}

\begin{conjecture}%[Sharp reverse Hardy inequality in \(\BMO_p\)]
    For every nonincreasing locally $p$-integrable function $f$ in $\Rplus$,
    \begin{equation}
        \Vert Hf\Vert_{\BMO_p}\geq\alpha_p\Vert f\Vert_{\BMO_p}.
    \end{equation}
\end{conjecture}

\thispagestyle{empty}

\end{document}